\documentclass[a4paper,11pt]{article}
\usepackage[utf8]{inputenc} 
\usepackage[english]{babel}
\usepackage[T1]{fontenc}   
\usepackage[normalem]{ulem}
\usepackage{mathpazo}
\usepackage{graphicx}
\usepackage{comment}
\usepackage{amsthm,amsmath,amsfonts,stmaryrd,mathrsfs,amssymb,mathtools}
\usepackage{mathtools}
\usepackage{bbm}
\usepackage{enumitem}
\usepackage{chemarrow}
\usepackage{tikz}
\usepackage[top=2.5cm, bottom=2.5cm, left=2cm, right=2cm]{geometry}
\usepackage{subfig}
\usepackage{multirow}
\usepackage{physics}
\usepackage{array}
\usepackage{bigints}
\usepackage{etoolbox}
\usepackage{mdframed}
\usepackage{color}
\usepackage{booktabs}
\usepackage{textgreek}
\usepackage[unicode,colorlinks=true,linkcolor=blue,citecolor=red,pdfencoding=auto,psdextra]{hyperref}
 \usepackage{dsfont}

\usepackage{enumitem}
\usepackage{amsmath}
\usepackage{upgreek}
\usepackage{commath}
\usepackage{amssymb}
\usepackage{xcolor}
\usepackage{amsthm}
\usepackage{csquotes}
\usepackage{titlesec}
\DeclareRobustCommand{\P}{\mathbb{P}}
\DeclareRobustCommand{\E}{\mathbb{E}}
\DeclareRobustCommand{\asy}{\underset {n\rightarrow \infty} \sim}
\DeclareRobustCommand{\to}{\xrightarrow[n\rightarrow \infty]{}}

\DeclareRobustCommand{\Z}{\mathbb{Z}}

\DeclareRobustCommand{\R}{\mathbb{R}}
\DeclareRobustCommand{\Cauchy}{\mathcal{C}_1}

\DeclareRobustCommand{\dist}{\xrightarrow[n \rightarrow \infty]{(d)}}
\DeclareRobustCommand{\prob}{\xrightarrow[n \rightarrow \infty]{(\P)}}

\DeclareRobustCommand{\T}{\mathcal{T}}

\DeclareRobustCommand{\dtv}{\mathrm{d}_{\mathrm{TV}}}

\titleformat{\subsection}[runin]
{\normalfont\bfseries}{\thesubsection}{1em}{}
\titleformat{\subsubsection}[runin]
{\bfseries}{\thesubsubsection}{1em}{}

\usepackage[normalem]{ulem}

\newmdtheoremenv{theorem}{Theorem}[section]
\newmdtheoremenv{proposition}[theorem]{Proposition}
\newtheorem{lemma}[theorem]{Lemma}

\newtheorem{corollary}[theorem]{Corollary}

\newtheorem{example}[theorem]{Example}

\DeclareSymbolFont{extraup}{U}{zavm}{m}{n}
\DeclareMathSymbol{\vardspade}{\mathalpha}{extraup}{81}
\DeclareMathSymbol{\varheart}{\mathalpha}{extraup}{86}
\DeclareMathSymbol{\vardiamond}{\mathalpha}{extraup}{87}
\DeclareMathSymbol{\varclub}{\mathalpha}{extraup}{84}

\makeatletter
\renewcommand*{\@fnsymbol}[1]{\ensuremath{\ifcase#1\or  \vardspade \or \varheart \or \vardiamond\or \varclub \or \bigstar \or
   \mathsection\or \mathparagraph\or \|\or **\or \dagger\dagger   \or \ddagger\ddagger \else\@ctrerr\fi}}
\makeatother

\author{
Igor Kortchemski 
\thanks{CNRS \& DMA, École normale supérieure, PSL University, 75005 Paris, France, \textsf{igor.kortchemski@math.cnrs.fr}
} 
\qquad
Leonard Vetter 
\thanks{ETH Zürich, \textsf{leonardvetter@web.de}
}   
}

\begin{document}
\title{Condensation in subcritical Cauchy Bienaymé trees}
\date{}
\maketitle

\begin{abstract}
The goal of this note is to study the geometry of large size-conditioned Bienaymé  trees whose offspring distribution is subcritical, belongs to the domain of attraction of a stable law of index $\alpha=1$ and satisfies a local regularity assumption. We show that a condensation phenomenon occurs: one unique vertex of macroscopic degree emerges, and its height converges in distribution to a geometric random variable. Furthermore, the height of such trees grows logarithmically in their size. Interestingly, the behavior of subcritical Bienaymée trees with $\alpha=1$ is quite similar to the case $\alpha \in( 1,2]$, in  contrast with the critical case. This completes the study of the height of heavy-tailed size-conditioned Bienaymé trees.

Our approach is to check that a random-walk one-big-jump principle due to Armend\'ariz \& Loulakis holds, by using local estimates due to Berger, combined with the previous approach to study subcritical Bienaymé trees with $\alpha>1$.
\end{abstract}

\begin{figure}[!ht]
\label{fig:plot}
\centering
\includegraphics[height=10cm]{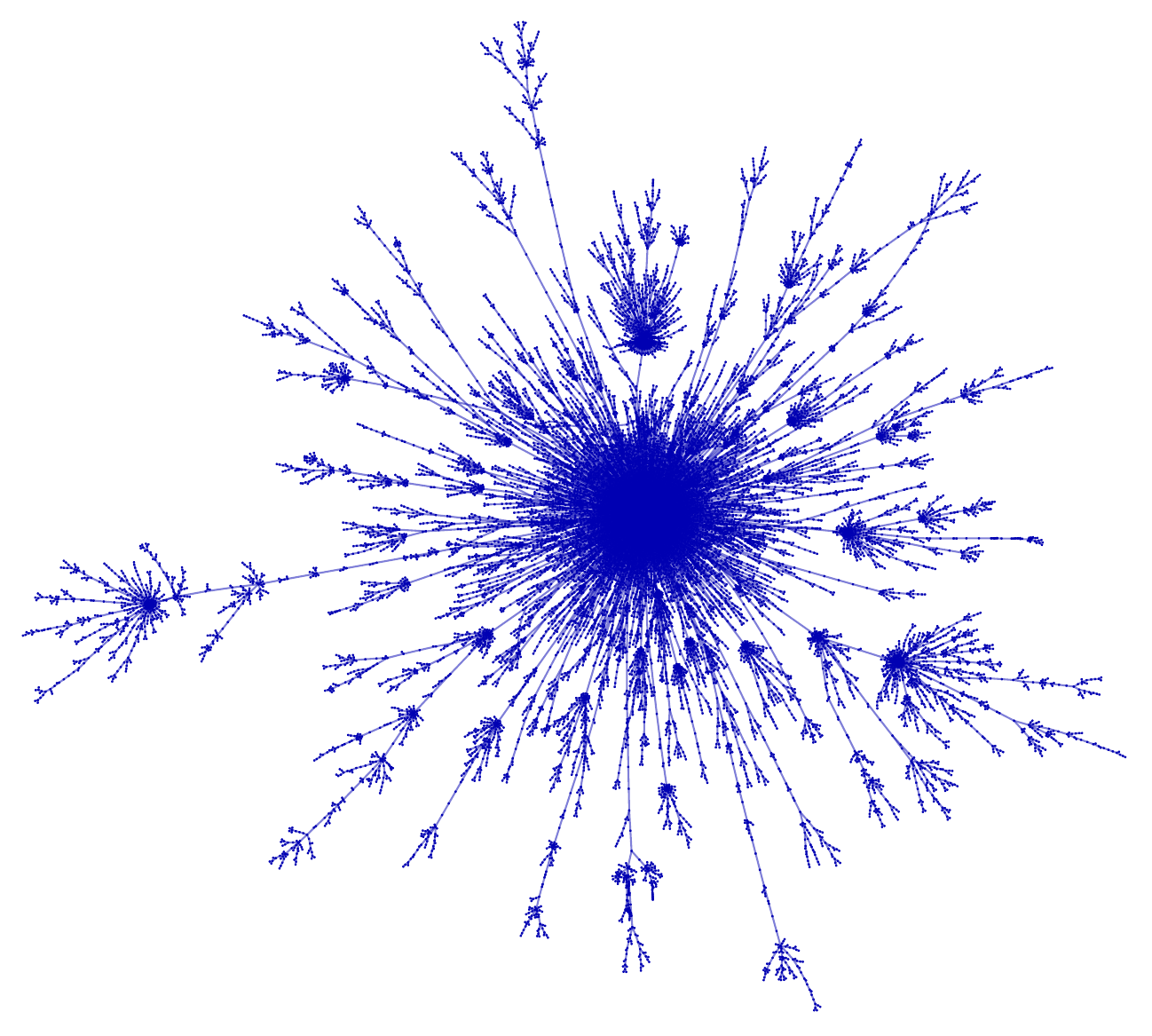}%
\caption{A simulation of a subcritical Cauchy Bienaymé tree with $20000$ vertices.
}
\end{figure}

\section{Introduction}
    
        The main purpose of this work is to  complete the study of the height of heavy-tailed size-conditioned Bienaymé trees (sometimes also called Bienaymé--Galton--Watson trees, or Galton--Watson trees in the literature) by considering an offspring distribution $\mu = (\mu_j:j\geq 0)$ such that 
         \begin{equation}
         \tag{$H^{\textrm{loc}}_\mu$}
         \label{eq:H_loc_mu}
            m \coloneqq \sum_{k\geq 0}k\mu_k < 1 \qquad  \textrm{and} \qquad  \mu_n \asy\frac{L(n)}{n^2},
        \end{equation}        
      where $L : \R_{+} \rightarrow \R^{*}_{+}$ is a slowly varying function, meaning that $ \lim_{x\rightarrow \infty}L(ax)/L(x) = 1$ for all $a> 0$.         The first condition amounts to saying that $\mu$ is subcritical, while the second condition implies that the offspring distribution is in the domain of attraction of a stable law of index $\alpha=1$ (i.e.~a Cauchy distribution).

        We denote by $\T_{n}$ a $\mu$-Bienaymé tree conditioned to have $n$ vertices  (we always implicitly restrict to those values of $n$ for which this event has non-zero probability). We let $\Delta(\mathcal{T}_n)$ be its maximum outdegree,  $\Delta^{2}(\mathcal{T}_n)$ be its second maximal outdegree, $\mathsf{H}_\Delta(\mathcal{T}_n)$ be the height of the vertex with maximal outdegree (first in lexicographical order if many, see below for details) and finally we let $\mathsf{H}(\T_{n})$ be the height of $\T_{n}$, i.e. the maximal graph distance of the root to any of its vertices.
        
              \begin{theorem}\label{thm:main}
              The following assertions hold.
              \begin{enumerate}
              \item[(i)] [condensation] We have
              $\displaystyle \frac{\Delta(\mathcal{T}_n)}{n(1-m)}\prob 1$ and $\displaystyle\frac{\Delta^{2}(\mathcal{T}_n)}{n}\prob 0.$
              \item[(ii)] [height of condensation vertex] For every $j \geq 0$ we have
              $$\P(H_{\Delta}(\mathcal{T}_n)=j) \to (1-m)m^{j}.$$
              \item[(iii)] [height of the tree] 
              \begin{enumerate}
              \item[(a)]               The following convergence holds in probability
                                 $$ \frac{\mathsf{H}(\mathcal{T}_n)}{\log(n)} \prob \frac{1}{\log(1/m)}.$$
              \item[(b)] The sequence $ \left({\mathsf{H}(\mathcal{T}_n)}- \frac{\log(n)}{\log(1/m)} \right)_{n \geq 1}$ is tight if and only if $\sum_{n \geq 1} ( n \log n) \mu_{n} <\infty$.
              \end{enumerate}
              \end{enumerate}
              \end{theorem} 
              This shows in particular that a condensation phenomenon occurs, in the sense that the maximal degree of $\T_{n}$ is comparable to the total size of the tree while the second largest degree is negligible compared to the total size of the tree. See Corollary \ref{cor:fluctuations} for the fluctuations of $\Delta(\mathcal{T}_n)$. We discuss some aspects of Theorem \ref{thm:main} after describing the context.

            \subsection{Context.}
The condensation phenomenon was  discovered by Jonsson and Stefánsson \cite{Jonsson_Stefansson_2011} for subcritical offspring distributions satisfying  $\mu_n \sim c/n^{1+\alpha}$ as $n\rightarrow \infty$ with $\alpha>1$ and $c>0$ (see also \cite[Sec.~19.6]{Jan12}). This was extended  in \cite{Kortchemski_2013} to subcritical offspring distributions satisfying $\mu_{n} \sim L(n)/n^{1+\alpha}$ with $L$ slowly varying and $\alpha>1$, where the same conclusions as those of Theorem \ref{thm:main} were established.

 In the critical case, when $\alpha>1$ there is no condensation phenomenon occurring. Indeed, for critical offspring distributions, when $\mu$ has finite variance (which requires $\alpha \geq 2$), the scaling limit of $ \mathcal{T}_{n}$ is Aldous' Brownian Continuum Random Tree \cite{Ald93}, and when $\mu$ has infinite variance (which requires $\alpha \in (1,2]$), the scaling limit of $ \mathcal{T}_{n}$ is the so-called the $\alpha$-stable tree \cite{Duq03}. 
 
 There is no condensation phenomenon neither for supercritical offspring distributions. Indeed, the study of size-conditioned Bienaymé trees with  supercritical offspring distribution is equivalent to the study of size-conditioned Bienamyé trees with critical finite variance offsprings by exponential tilting (see Sec.~\ref{ssec:tilting} below), which are covered by Aldous' previously mentioned  result.
        
        The so-called Cauchy case where $\alpha=1$ has only been considered quite recently in the critical case \cite{Kortchemski_Richier_2018a}, motivated by applications to random planar maps: a condensation phenomenon occurs, although on a slightly different scale (the unique vertex of maximal degree is of order $o(n)$, yet dominates the degrees of the other vertices). Limit theorems for the height of such trees have been established in \cite{addarioberry2023criticaltreesshortfat}.
        
In the recent years, it has been realized that Bienaymé trees in which a condensation phenomenon occurs code a variety of random combinatorial structures such as random planar maps \cite{Add19,JS15,Ric17,AM22,AFS24}, outerplanar maps \cite{SS17}, supercritical percolation clusters of random triangulations \cite{CK15}, random permutations \cite{BBFS20}, parking on random trees \cite{CC23} or minimal factorizations \cite{FK17}. See \cite{Stu16} for a combinatorial  framework and further examples. These applications are one of the motivations for the study of the fine structure of such large conditioned Bienaymé trees.

 \subsection{Comments on Theorem \ref{thm:main}.}   The goal of this paper is to cover the case where $\mu$ is subcritical and $\alpha=1$, which is the last missing case for heavy-tailed offspring distributions satisfying a local regularity assumption (see Sec.~\ref{ssec:tilting} below). The results of Theorem \ref{thm:main} (i), (ii) and (iii) (a) are the same as for subcritical offspring distributions with $\alpha>1$ \cite{Kortchemski_2013}. However, an interesting new phenomenon appears in the case $\alpha=1$: indeed, the sequence $ \left({\mathsf{H}(\mathcal{T}_n)}- {\log(n)}/{\log(1/m)} \right)_{n \geq 1}$ is not always tight (in contrast with the case $\alpha>1$ where it is always tight). The reason is that when $\alpha>1$ we always have $\sum_{n \geq 1} (n \log n) \mu_{n}  <\infty$, but in the case $\alpha=1$ we may have $\sum_{n \geq 1} (n \log n) \mu_n = \infty$; take for example $\mu_{n} \sim \frac{c}{(\log n)^{1 + \beta} n^{2}}$ with $\beta \in (0,1]$. See Proposition \ref{prop:ex} for the second order term for the magnitude of ${\mathsf{H}(\mathcal{T}_n)}$ in this particular case.

 It is also interesting to note that for subcritical offspring distributions, the behavior of large size-conditioned Bienaymé trees turns out to be similar for $\alpha=1$ and $\alpha \in (1,2]$, while for critical offspring distributions, the behavior of large size-conditioned Bienaymé trees for $\alpha=1$ and $\alpha \in (1,2]$ is quite different (there are no non-trivial scaling limits in the  case $\alpha=1$). 
    
   \subsection{Main ideas.} Using \cite{Ber17}, we show that the one-big jump principle of  \cite{AL11} used in the case $\alpha>1$ in \cite{Kortchemski_2013} holds when $\alpha=1$ as well. Theorem \ref{thm:main} (i) and (ii) then follow as in  \cite{Kortchemski_2013} from this one-big jump principle. However the proof of Theorem \ref{thm:main} (iii) concerning the height of $ \mathcal{T}_{n}$ requires different estimates. Indeed, the estimate $\P( \mathsf{H}( \mathcal{T}) \geq n) \sim c \cdot m^{n}$, where $ \mathcal{T}$ is an unconditioned $\mu$-Bienaymé tree,  used in    \cite{Kortchemski_2013}  for $\alpha>1$, is not true in general for $\alpha=1$.
   
    \subsection{Outline.} We first introduce Bienaymé trees and their associated random walks in Section \ref{sec:B}.  The main one big jump principle for conditioned random walks is presented in Section \ref{sec:jump}, and limit theorems for subcritical Cauchy Bienaymé trees are proved in Section \ref{sec:condensation}.

   \subsection{Acknowledgments.} We would like to thank the referees for their very careful reading and their helpful comments.
            
    \section{Bienaymé trees and random walks}
  \label{sec:B}
    
\subsection{Bienaymé trees.}            We consider plane trees, which are also sometimes called rooted ordered trees (see \cite[Sec.~1]{Le_Gall_2005} for definitions and background). For every plane tree $\tau$ and every vertex $v \in \tau$, we denote by $k_v(\tau)$ the outdegree (or number of children) of $v$, so that $\Delta(\tau)=\max_{v \in \tau} k_{v}(\tau)$ is the maximal outdegree of $\tau$.
            
            Given a probability distribution   $\mu=(\mu_{n})_{n \geq 0}$, on the nonnegative integers $\Z_{+}$, we denote by $\P_{\mu}$ the law of a Bienaymé tree with offspring distribution $\mu$. It satisfies for every finite tree $\tau$ the identity
                            $$\P_\mu(\tau) = \prod_{u\in \tau} \mu_{k_u(\tau)}.$$

For every $n \geq 1$, we denote by $ \mathcal{T}^{\mu}_{n}$ a Bienaymé tree with offspring distribution $\mu$ conditioned to have $n$ vertices (we always implicitly restrict to those values of $n$ for which this event has non-zero probability).
 
\subsection{Exponential tilting.}
\label{ssec:tilting}

Exponential tilting is a useful tool which allows to tune the mean of a Bienaymé tree without changing the law of the associated size-conditioned Bienaymé tree. It is essentially due to Kennedy \cite{Ken75} (in the context of branching processes).

\begin{proposition}[Kennedy]
\label{prop:tilting}
Let $\mu$ and $\nu$ be two probability distributions on $\Z_{+}$ such that for certain $a,\lambda>0$ we have $\mu_k=a \lambda^{k} \nu_k$ for every $k \geq 0$. Then $\mathcal {T}^{\mu}_n$ and $\mathcal {T}^{\nu}_n$ have the same distribution.
\end{proposition}

If $\mu$ and $\nu$ are related as in Proposition \ref{prop:tilting}, we say that they are equivalent.  It is a simple matter to characterize the class of offspring distributions $\mu$ which are equivalent to a critical offspring distribution. Indeed, let $F_{\mu}(x)=\sum_{k=0}^{\infty} \mu_k x^{k}$ be the generating function of $\mu$ and denote by $\rho$ its radius of convergence. Then for every $\lambda \in (0,\rho)$, the offspring distribution $\mu_{\lambda}$ with generating function $F_{\mu_{\lambda}}(x)=F_{\mu}(\lambda x)/F_{\mu}(\lambda)$ is equivalent to $\mu$ and has expectation $\lambda F'_{\mu}(\lambda)/F_{\mu}(\lambda)$. Thus $\mu$ is equivalent to a critical offspring distribution if and only if $ \lim_{\lambda \rightarrow  \rho} \lambda F'_{\mu}(\lambda)/F_{\mu}(\lambda) \geq 1$. In particular, observe that a supercritical offspring distribution (i.e.~that has mean greater than $1$) is always equivalent to a critical offspring distribution. 

As a consequence, if $\nu$ is a heavy-tailed offspring distribution of the form $\nu_n \sim{L(n)}/{n^{1+\alpha}}$ with $\alpha \geq 0$ and $L$ slowly varying then $\nu$ is equivalent to one of the following:
\begin{enumerate}
\item[(a)] a critical offspring distribution $\mu$ with finite variance. In this case the asymptotic behavior of $\T^{\mu}_{n}$ is known \cite{Ald93}.
\item[(b)] a critical offspring distribution $\mu$ with infinite variance and $\alpha \in (1,2]$. In this case the asymptotic behavior of $\T^{\mu}_{n}$ is known \cite{Duq03}.
\item[(c)] a critical offspring distribution $\mu$ with $\alpha=1$. In this case the asymptotic behavior of $\T^{\mu}_{n}$ is known \cite{addarioberry2023criticaltreesshortfat}.
\item[(d)] a subcritical offspring distribution with $\alpha>1$. In this case the asymptotic behavior of $\T^{\mu}_{n}$ is known \cite{Kortchemski_2013}.
\item[(e)] a subcritical offspring distribution $\mu$ with $\alpha=1$. This case is last missing case considered in this note.
\end{enumerate}
Observe that when $\nu$ is supercritical (which includes the case $\alpha<1$ since $\nu$ has then infinite mean) we are in case (a).

\subsection{Random walks.}
\label{ssec:RW}An important tool to study Bienaymé trees is the use of random walks. Given an offspring distribution $\mu$, let $X$ be a random variable with law given by $\P(X = i) = \mu_{i+1}$ for $i \in \Z_{\geq -1}$. Let $(X_i)_{i\geq 1}$ be i.i.d. random variables distributed as $X$ and let $(W_{k})_{k \geq 0}$ be the random walk defined by $W_{0}=0$ and $W_k= \sum_{i = 1}^k X_i$ for $k \geq 1$.

The key connection between Bienaymé trees and random walks is that is it possible to bijectively code plane trees by the so-called \L ukasiewicz path, in such a way that the \L ukasiewicz path of a Bienaymé tree has the law of $(W_{k})_{k \geq 0}$ stopped at the first hitting time of the negative integers. As a consequence, studying  $ \mathcal{T}^{\mu}_{n}$ is equivalent to studying the random walk $(W_{k})_{k \geq 0}$ conditioned on hitting the negative integers at time $n$ (``excursion''-type conditioning). In turn, this is equivalent to studying the random walk $(W_{k})_{k \geq 0}$ conditioned on hitting $-1$ at time $n$ (``bridge''-type conditioning) using the Vervaat transform. One of the key implications is that the collection of outdegrees minus $1$ of $ \mathcal{T}_{n}$ has the same law as the jumps of $(W_{k})_{0 \leq k \leq n}$ under the conditional probability $\P(\, \cdot \mid W_{n}=-1)$, see \cite[Sec.~6.1]{Pit06} for background.

From now on, assume that $\mu$ satisfies \eqref{eq:H_loc_mu}. We  need to introduce two sequences related to the asymptotic behavior of $(W_{n})_{n \geq 1}$.  Observe that $\E[X] = m-1 <0$. Let $(a_n:n\geq 1)$ and $(b_n:n\geq 1)$ be sequences such that
        \begin{equation}
            n\P(X \geq a_n) \to 1, \hspace{1cm} b_n=n\E[X\mathbb{1}_{\abs{X}\leq a_n}],
        \end{equation}
Then the following convergence holds in distribution
        \begin{equation}\label{eq:conv to C_1}
            \frac{X_1+ \cdots +X_n-b_n}{a_n} \dist \Cauchy,
        \end{equation}
        where $\Cauchy$ is a random variable, with Laplace transform given by $\E[e^{-\lambda \Cauchy}] = e^{\lambda \log{\lambda}}$ for $\lambda>0$, see \cite[Chap. IX.8 and Eq. (8.15) p.315]{feller-vol-2}.   For this reason, we often call $(a_n:n\geq 1)$ the scaling sequence and $(b_n:n\geq 1)$ the centering sequence.  The random variable $\Cauchy$ is an asymmetric Cauchy random variable with skewness $1$. In addition $(a_{n})$ and $(b_{n})$ are regularly varying sequences of index $1$.
        
%        The following bound for slowly varying functions are useful (see \cite[Theorem 1.5.6]{Bingham_Goldie_Teugels_1987} for a proof).
%
%\begin{proposition}[Potter bounds]
%\label{prop:Potter}
%Let $\ell : \R_{+} \rightarrow \R^{*}_{+}$ be a slowly varying function. For every $A>1$ and $\delta>0$ there exists $M>0$ such that for every $x \geq M$ and $y \geq M$ we have
%$$ \frac{\ell(y)}{\ell(x)} \leq  A \max \left(  \left(  \frac{y}{x} \right)^{\delta}, \left(  \frac{y}{x} \right)^{-\delta}\right).$$
%\end{proposition}
%As a consequence, for every $ \varepsilon>0$ there exists $C>0$ such that for every $x \geq 1$ we have
%\begin{equation}
%\label{eq:growth}
% \frac{1}{C}  x^{-\varepsilon} \leq \ell(x) \leq C x^{\varepsilon}.
%\end{equation}
        We will also need to introduce an auxiliary slowly varying function. For every $n \geq 1$ set $\ell^\star(n) \coloneqq \sum_{k=n}^\infty L(k)/k$ for $n \geq 1$, which is a finite quantity since $\mu$ has finite mean.
        
        \begin{lemma}
        \label{lem:bnan} The following assertions hold as $n\rightarrow \infty$.
        \begin{enumerate}
        \item[(i)]  The function $\ell^{\star}$ is slowly varying, satisfies $\ell^{\star}(n) \rightarrow 0$ and  $L(n) = o(\ell^\star(n))$.
        \item[(ii)] We have $b_{n}+ n(1-m) \sim - n \ell^{\star}(a_{n})$ and $b_n \sim - n(1-m)$.
        \item[(iii)] We have  $a_n = o(n)$.
        \end{enumerate}
        \end{lemma}

\begin{proof}
 For the first assertion, by definition we clearly have $\ell^{\star}(n) \rightarrow 0$ as $n \rightarrow \infty$. The two other properties follow from  \cite[Proposition 1.5.9b]{Bingham_Goldie_Teugels_1987}.
 
 For (ii),  write
$b_{n}=n(\E[X]-\E[X \mathds{1}_{\abs{X}>a_n}]) =n(m-1)-n\E[X \mathds{1}_{\abs{X}>a_n}]
$
and observe that by assumption \eqref{eq:H_loc_mu} we have $n\E[X \mathds{1}_{\abs{X}>a_n}] \sim n \ell^\star(a_n)$, which gives the first asymptotic estimate. The second one follows from the fact that $\ell^\star(a_n) \rightarrow 0$ because $a_{n} \rightarrow \infty$.

Finally, for the last assertion, we combine the fact that $a_{n} \sim n L(a_{n})$ (by definition of $a_{n}$) with the fact that $\ell^{\star}(n) \rightarrow 0$ and write for $n$ sufficiently large
$$ \frac{a_{n}}{n} \leq  \frac{n L(a_{n})}{n \ell^{\star}(a_{n})},$$
which converges to $0$ since $L(n) = o(\ell^\star(n))$.
\end{proof}

In particular, \eqref{eq:conv to C_1} and Lemma  \ref{lem:bnan} imply that
   \begin{equation}\label{eq:proba}
            \frac{X_1+ \cdots +X_n}{n (m-1)} \prob 1.
        \end{equation}
        Also observe that by Lemma \ref{lem:bnan} (i) and (ii), it is not true that $(W_{n}- \E[W_{n}])/a_{n})$ converges in distribution as $n \rightarrow \infty$: in contrast with the case $\alpha>1$ some care is needed in handling the centering term.

\begin{example}
\label{ex:anl}
In some particular cases we can explicitly determine the asymptotic behavior of $a_n$ and $\ell^\star$: 
\begin{itemize}
\item[(i)] If $L(n)\sim c  / \log(n)^{1+\beta}$ with $c,\beta>0$, then $\ell^\star(n) \sim c/(\beta \log(n)^{\beta})$ and $a_n \sim   n L(n)$ since $\log(a_n) \sim \log(n)$.
\item[(ii)] Set $\log_{(1)}(x)=\log(x)$ and for every $k \geq 1$ define recursively $\log_{(k+1)} (x) = \log(\log_{(k)}(x))$. For
 \[
 L(n) \quad \sim \quad  \frac{c}{(\log_{(k)}(n))^{2}} \prod_{i=1}^{k-1} \frac{1}{\log_{(i)}(n) } \]
 with $k \geq 2$ and $c>0$, we have $\ell^\star(n) \sim c/\log_{(k)}(n)$ and $a_n \sim   n L(n)$  since $\log(a_n) \sim \log(n)$. 
\item[(iii)]   For $L(n) \sim  c   \exp(-\log(n)^{\beta}) $ with $\beta \in (0,1)$ and $c>0$, we have $\ell^\star(n) \sim   c/\beta \cdot \log(n)^{1-\beta} \exp(-\log(n)^{\beta})$. The asymptotic behavior of $a_n$ depends on the value of $\beta$, for example when $\beta<1/2$ we have $a_n \sim c   n\exp(-\log(n)^{\beta}) $ and when $1/2 \leq \beta <2/3$ we have $a_n \sim c   n \exp(-\log(n)^{\beta} + \beta \log(n)^{2 \beta-1})$.
\end{itemize}
\end{example}

\section{A one big jump principle for conditioned random walks}
\label{sec:jump}

Here we consider an offspring distribution $\mu$ satisfying \eqref{eq:H_loc_mu} and denote by  $(W_{k})_{k \geq 0}$   the random walk defined in Sec.~\ref{ssec:RW}. The key ingredient that enables us to use the results of \cite{Kortchemski_2013} is a one-big jump principle for the random walk under the conditional probability $\P(\cdot\mid W_n=-1)$. 

In the case $\alpha>1$, this was obtained in \cite{Kortchemski_2013} thanks to a general result due to Armend\'ariz \& Loulakis \cite{AL11}, using local estimates for random walks obtained in \cite{DDS08}. In the case $\alpha=1$, we show that we can still use the result of Armend\'ariz \& Loulakis \cite{AL11}, using instead local estimates for random walks obtained by Berger \cite{Ber17}.

 For every integer $n \geq 1$, let $V_{n}$ be the index of the first maximal jump of $(W_{k})_{0 \leq k \leq n}$ defined by
            \begin{equation*}
                V_n \coloneqq \inf\Big\{1\leq j\leq n:X_j=\max\{X_i:1\leq i\leq n\}\Big\}.
            \end{equation*}
            Then we define $(X_1^{(n)},...,X_{n-1}^{(n)})$ to be the random variable distributed as the law of $(X_1,...,X_{V_n-1},X_{V_n+1},...,X_n)$ under $\P(\cdot\mid W_n=-1)$. The following theorem states that once the first maximal jump of $(W_k)_{0 \leq k \leq n}$ under $\P(\cdot\mid W_n=-1)$ is removed, the remaining increments behave asymptotically like i.i.d. random variables.
            \begin{theorem}
            \label{th:asy iid}
                We have
                \begin{equation*}
                    d_{\mathrm{TV}} \left((X_i^{(n)}:1\leq i\leq n-1),(X_i:1\leq i\leq n-1)\right) \to 0
                \end{equation*}
                where $d_{\mathrm{TV}}$ denotes the total variation distance on $\mathbb{R}^{n-1}$.
            \end{theorem}
            
\begin{proof}
To simplify notation, set $\gamma=1-m$. For every $n \geq 1$, set $\overline{X}_{n}=X_{n}+\gamma$ and $\overline{W}_{n}=W_{n}+\gamma n$, so that $(\overline{W}_{n})_{n \geq 0}$ is a centered random walk. Similarly, set  $(\overline{X}_1^{(n)}, \ldots ,\overline{X}_{n-1}^{(n)})=(X_1^{(n)}+\gamma, \ldots ,X_{n-1}^{(n)}+\gamma)$. Fix $\varepsilon \in (0,\gamma)$. We check that we can apply Theorem 1 in \cite{AL11}, with $\mu$ being the law of $\overline{X}_{1}$, $\Delta=(0,1]$, $q_{n}=\varepsilon n$ and $x=\gamma n-1$. This will indeed imply that 
$$
                    d_{\mathrm{TV}}\left((\overline{X}_i^{(n)}:1\leq i\leq n-1),(\overline{X}_i:1\leq i\leq n-1)\right) \to 0,
$$
giving the desired result.

In order to apply  Theorem 1 in \cite{AL11}, we check that condition (2.6) there holds with $d_{n}=\varepsilon n$ and that condition (3.3) there holds also with $\ell_{n}=\varepsilon n$. Recall the definition of $\ell^\star$ introduced just before Lemma \ref{lem:bnan}.

\emph{Condition (2.6).} We need to check that
$$
\lim_{n \rightarrow \infty} \sup_{x \geq \varepsilon n} \left| \frac{\P(\overline{W}_{n} \in (x,x+1])}{ n \P(\overline{X}_{1} \in (x,x+1])} -1 \right|=0,
$$
or, equivalently, setting $m_{n}=b_{n}+\gamma n$,
\begin{equation}
\label{eq:26} 
\lim_{n \rightarrow \infty} \sup_{x \geq \varepsilon n} \left| \frac{\P({W}_{n}-b_{n} \in (x-m_{n},x-m_{n}+1])}{ n \P({X}_{1} \in (x-\gamma,x-\gamma+1])} -1 \right|=0.
\end{equation}
We claim that $|x-m_{n}|/a_{n} \rightarrow \infty$  uniformly in $x \geq \varepsilon n$. Indeed, by Lemma \ref{lem:bnan} (ii), we have $m_{n} \sim -n \ell^{\star}(a_{n})$, so that $|m_{n}|=o(n)$ since $\ell^{\star}(n) \rightarrow 0$ as $n \rightarrow \infty$. It follows that $|x-m_{n}| \sim x $  uniformly in $x \geq \varepsilon n$, and we get our claim since $a_n = o(n)$ (Lemma \ref{lem:bnan} (iii)). This puts us in position to use Theorem 2.4 in \cite{Ber17} (in the reference we take $\alpha= 1$ and $x=-\lfloor b_n \rfloor -1$), which gives that 
\begin{equation}
\label{eq:ber}\lim_{n \rightarrow \infty} \sup_{x \geq \varepsilon n} \left| \frac{\P({W}_{n}-b_{n} \in (x-m_{n},x-m_{n}+1])}{ n \P({X}_{1} \in(x-m_{n},x-m_{n}+1])} -1 \right|=0.
\end{equation}
Now observe that since $L$ is slowly varying at infinity, for any sequence $\delta_n \rightarrow 0$ of positive real numbers and any sequence of integers $z_n \rightarrow \infty$ we have the convergence $\sup_{(1-\delta_n) z_n \leq y \leq (1+\delta_n) z_n} L(y)/L(z_n) \rightarrow 1$ (this follows e.g. from the representation theorem for slowly varying functions). Using \eqref{eq:H_loc_mu} and $|x-m_{n}| \sim x$ uniformly in $x\geq \epsilon n$ together with the assumption that $L(\cdot)$ is slowly varying at infinity, we get
$$ \lim_{n \rightarrow \infty} \sup_{x \geq \varepsilon n} \left| \frac{ \P({X}_{1} \in(x-m_{n},x-m_{n}+1])}{  \P({X}_{1} \in(x-\gamma,x-\gamma+1])} -1 \right|=0.$$
Combined with \eqref{eq:ber} we get \eqref{eq:26}.

\emph{Condition (3.3).} Set $\widehat{b}_{n}= n \ell^{\star}(a_{n})$. We check that $(\overline{W}_{n}/\widehat{b}_{n})_{n \geq 1}$ is tight and that for every $L>0$ we have
\begin{equation}
\label{eq:33} \sup_{x \geq \varepsilon n} \sup_{|y| \leq L \widehat{b}_{n}} \left| 1- \frac{\P(\overline{X}_{1} \in (x-y,x-y+1])}{\P(\overline{X}_{1} \in (x,x+1])} \right| \to 0.
\end{equation}
To check tightness, write
$$ \frac{\overline{W}_{n}}{\widehat{b}_{n}} = \frac{X_{1}+ \cdots+X_{n}-b_{n}}{a_{n}} \cdot \frac{a_{n}}{\widehat{b}_{n}}+ \frac{b_{n}+\gamma n}{\widehat{b}_{n}},$$
which implies tightness since $a_{n}/\widehat{b}_{n} \sim L(a_{n})/\ell^{\star}(a_{n}) \rightarrow 0$ and $ ({b_{n}+\gamma n})/{\widehat{b}_{n}} \rightarrow -1$ by Lemma \ref{lem:bnan} (ii). 
The convergence \eqref{eq:33} readily follows from the fact that $\widehat{b}_{n}=o(n)$ and the fact that for every $\delta>0$ and for every positive sequence $(\eta_{n})$ of real numbers such that $\eta_{n}=o(n)$ we have
$$\sup_{x \geq \varepsilon n} \sup_{|u| \leq \eta_{n}} \left| \frac{L(x+u)}{L(x)}-1\right|\to 0.$$
This can e.g.~be seen using the representation theorem for slowly varying functions \cite[Theorem 1.3.1]{Bingham_Goldie_Teugels_1987}. This completes the proof.
\end{proof}

This establishes the same one-big jump principle as when $\mu$ is subcritical and $\mu(n) \sim L(n)/n^{1+\alpha}$ with $\alpha>1$, which is  Theorem 2.1 in \cite{Kortchemski_2013}.

\section{Condensation in subcritical Cauchy Bienaymé trees}
\label{sec:condensation}

We are now ready to establish our results for subcritical Cauchy Bienaymé trees.    As before, we consider an offspring distribution $\mu$ satisfying \eqref{eq:H_loc_mu} and denote by  $(W_{k})_{k \geq 0}$   the random walk defined in Sec.~\ref{ssec:RW}. Recall that $ \mathcal{T}_{n}$ denotes a Bienaymé tree with offspring distribution $\mu$ conditioned to have $n$ vertices.

\subsection{Condensation phenomenon.} Theorem \ref{thm:main} (i) and (ii) are proved  in the same way as Theorem 1 and Theorem 2 are proved in \cite{Kortchemski_2013} in the case where $\mu$ is subcritical and $\mu(n) \sim L(n)/n^{1+\alpha}$ with $\alpha>1$. Indeed, for Theorem \ref{thm:main} (i), by combining the one big jump principle (Theorem \ref{th:asy iid}) with the fact that the collection of outdegrees minus $1$ of $ \mathcal{T}_{n}$ has the same law as the collection of jumps of $(W_{k})_{0 \leq k \leq n}$ under the conditional probability $\P(\, \cdot \mid W_{n}=-1)$, we get
\begin{equation}
\label{eq:dtvDelta}
\dtv \left( \Delta(\mathcal{T}_n), - (X_{1}+ \cdots +X_{n-1}) \right) \to 0,
\end{equation}
which by \eqref{eq:proba} yields the first convergence of Theorem \ref{thm:main} (i). Also,  
$$\dtv \left( \Delta^{2}(\mathcal{T}_n)-1, \max(X_{1}, \ldots,X_{n-1}) \right) \to 0,$$
which using the fact that $\P(X_{1} \geq u) \sim L(u)/u$ as $u \rightarrow \infty$ implies that $ \max(X_{1}, \ldots,X_{n-1})/a_{n}$ converges in distribution to a random variable $Y$ with law given by $\P(Y \leq u)=\exp(-1/u)$ for $u>0$. Since $a_{n}=o(n)$ this implies the second convergence of Theorem \ref{thm:main} (i).

Theorem \ref{thm:main} (ii) is established in the exact same way Theorem 2 in \cite{Kortchemski_2013} is proved, taking as input the one big jump principle (Theorem \ref{th:asy iid} in our case $\alpha=1$).

Also, by combining \eqref{eq:conv to C_1} with \eqref{eq:dtvDelta} we immediately get the following fluctuations for $\Delta(\mathcal{T}_n)$:

\begin{corollary}
\label{cor:fluctuations}
We have
$$            \frac{\Delta(\mathcal{T}_n)+b_n}{a_n} \dist -\Cauchy.$$
\end{corollary}

As suggested by an anonymous referee, it would be very interesting to extend this central limit theorem to a local limit theorem, motivated by Stufler's result \cite[Lemma 2.2]{Stu20} for $\alpha>1$ (the proof of \cite[Lemma 2.2]{Stu20} strongly relies on the fact that $\alpha>1$, so new input is needed). 
 
\subsection{Height of $\mathcal{T}_{n}$.} In  \cite{Kortchemski_2013}, the proof of the fact that $\mathsf{H}(\mathcal{T}_{n})/ \log(n) \rightarrow 1/\log(1/m)$ in probability when $\mu$ is subcritical and $\mu(n) \sim L(n)/n^{1+\alpha}$ with $\alpha>1$ uses the asymptotic estimate $\P( \mathsf{H}( \mathcal{T}) \geq n) \sim c \cdot m^{n}$, where $ \mathcal{T}$ is an unconditioned $\mu$-Bienaymé tree. However, this estimate is not true in general when $\alpha=1$. Indeed, by \cite[Theorem 2]{HSV67}, this asymptotic estimate holds if and only if $\sum_{n \geq 1} ( n \log n) \mu_n <\infty$; when this sum is infinite we have $\P( \mathsf{H}( \mathcal{T}) \geq n) /m^{n} \rightarrow 0$ as $n \rightarrow \infty$. Observe that in the case $\alpha=1$, as was already mentioned we can have $\sum_{n \geq 1} ( n \log n) \mu_{n} =\infty$.

In the case $\alpha=1$, we use the same idea as in \cite{Kortchemski_2013}, which consists in using the fact that, roughly speaking, the trees grafted on the vertex of maximal degree of $ \mathcal{T}_{n}$ are asymptotically independent $\mu$-Bienaymé trees, combined with a bound on $\P( \mathsf{H}( \mathcal{T}) \geq n)$.

We need to introduce some notation. For every finite plane tree $\tau$, denote by $u_{\star}(\tau)$ the vertex of maximal degree (first in lexicographical order if not unique) and for every $1 \leq i \leq \Delta(\tau)$, let $\tau_{i}$ be the tree of descendants of the $i$-th child of $u_{\star}(\tau)$. For $i > \Delta(\tau)$ we set $\tau_{i}=\varnothing$. For $1 \leq j \leq k$, we let $ [\tau]_{j,k}$ be the forest defined by $[\tau]_{j,k}= \{\tau_{i} : j \leq i \leq k \}$. Finally for every $i \geq 1$, denote by $\mathcal{F}^{ i}$ the forest of $i$ independent $\mu$-Bienaymé trees.

\begin{proposition}
\label{prop:forests}For every $\delta \in [0,1-m)$, we have
$\dtv \left(  \left[  \mathcal{T}_{n}  \right]_{1,\lfloor \delta n\rfloor}, \mathcal{F}^{ \lfloor \delta n\rfloor}  \right) \rightarrow 0$ as $n \rightarrow \infty$.
\end{proposition}

This is established in the exact same way Corollary 2.7 is proved in  \cite{Kortchemski_2013},  again taking as input the the one big jump principle (Theorem \ref{th:asy iid} in our case $\alpha=1$).

For every $n \geq 0$ set 
$Q_n = \P(H(\mathcal{T})\geq n)$,
where $ \mathcal{T} $ is an (unconditioned) $\mu$-Bienaymé tree.

\begin{proposition}
\label{prop:height}
Let $(h_{n})$ be a sequence of positive real numbers such that $h_{n} \rightarrow \infty$. The following assertions hold.
\begin{enumerate}
\item[(i)] If $n Q_{ h_{n}} \rightarrow \infty$ then $\P(\mathsf{H}(\mathcal{T}_n) \geq  h_{n}) \rightarrow 1$ as $n \rightarrow \infty$.
\item[(ii)] If $n Q_{ h_{n}} \rightarrow 0$ then $\P(\mathsf{H}(\mathcal{T}_n) < h_{n}) \rightarrow 1$ as $n \rightarrow \infty$.
\end{enumerate}
\end{proposition}

\begin{proof}
We mimic the proof of Theorem 4 in \cite {Kortchemski_2013}. For every tree $\tau$, denote by $\mathsf{H}_{\star}(\tau)$ the height of the forest $ [\tau]_{1,\Delta(\tau)}$. To simplify notation, we set $\Delta_{n}=\Delta(\mathcal{T}_{n})$. By Theorem \ref{thm:main} (ii), it is enough to show (i) and (ii) with $\mathsf{H}(\mathcal{T}_n) $ replaced by $\mathsf{H}_{\star}(\mathcal{T}_n) $. 

We start with (i). Observe that by Theorem \ref{thm:main} (i), setting $\delta=(1-m)/2$, we have $ \P(\Delta_{n} \geq  \lfloor \delta n \rfloor) \rightarrow 1$ as $n \rightarrow \infty$. Thus, using Proposition \ref{prop:forests},
$$\P(\mathsf{H}_{\star}(\mathcal{T}_n) <  h_{n}) \leq  \P(\mathsf{H}( [\mathcal{T}_{n}]_{1, \lfloor \delta n \rfloor}) < h_{n})+o(1) =   \P(\mathsf{H}( \mathcal{F}^{ \lfloor \delta n\rfloor}) < h_{n})+o(1).$$
But $\P(\mathsf{H}( \mathcal{F}^{ \lfloor \delta n\rfloor}) < h_{n})=(1-Q_{h_{n}})^{\lfloor \delta n \rfloor} \rightarrow 0$
since $n Q_{ h_{n}} \rightarrow \infty$.

For (ii), write
$$\P(\mathsf{H}_{\star}(\mathcal{T}_n) \geq h_{n}) \leq \P(\mathsf{H}_{\star}([\mathcal{T}_n]_{1,\lfloor \Delta_{n}/2\rfloor})\geq h_{n})+\P(\mathsf{H}_{\star}([\mathcal{T}_n]_{\lceil \Delta_{n}/2\rceil,\Delta_{n}})\geq h_{n}).$$
Since $[\mathcal{T}_n]_{\lceil \Delta_{n}/2\rceil,\Delta_{n}}$ and $[\mathcal{T}_n]_{1,\Delta_{n}-\lfloor \Delta_{n}/2\rfloor}$ have the same distribution, it suffices to show that $ \P(\mathsf{H}_{\star}([\mathcal{T}_n]_{1,\lfloor \Delta_{n}/2\rfloor})\geq h_{n}) \rightarrow 0$ as $n \rightarrow \infty$. Set $\delta=2(1-m)/3$. By Theorem \ref{thm:main} (i), we have $ \P(\lfloor \Delta_{n}/2\rfloor \leq \lfloor \delta n \rfloor) \rightarrow 1$ as $n \rightarrow \infty$.  Thus, using Proposition \ref{prop:forests},
$$ \P(\mathsf{H}_{\star}([\mathcal{T}_n]_{1,\lfloor \Delta_{n}/2\rfloor})\geq h_{n}) \leq  \P(\mathsf{H}( [\mathcal{T}_{n}]_{1, \lfloor \delta n \rfloor} \geq  h_{n})+o(1) = \P(\mathsf{H}( \mathcal{F}^{ \lfloor \delta n\rfloor}) \geq h_{n})+o(1).$$
But $ \P(\mathsf{H}( \mathcal{F}^{ \lfloor \delta n\rfloor}) \geq h_{n})=1-\P(\mathsf{H}( \mathcal{F}^{ \lfloor \delta n\rfloor}) <  h_{n})=1-(1-Q_{h_{n}})^{\lfloor \delta n \rfloor} \rightarrow 0$
since $n Q_{ h_{n}} \rightarrow 0$.
This completes the proof.
\end{proof}

In order to apply Proposition \ref{prop:height} we will use the following bounds on $Q_{n}$.

\begin{lemma}
\label{lem:bounds} The following assertions hold.
\begin{enumerate}
\item[(i)] There is a   function $\ell : (0,1] \rightarrow \R_{+}^{*}$ slowly varying at $0$ such that $\ell(x) \rightarrow 0$ as $x \rightarrow 0$ and $Q_{n+1} = Q_n\big(m-\ell(Q_n)\big)$ for every  $n \geq 0$.
\item[(ii)] For every $\eta \in (0,m)$, for every $n$ sufficiently large  we have $(m-\eta)^{n} \leq Q_{n} \leq m^{n}.$
\end{enumerate}
\end{lemma} 
\begin{proof}
 Let $G_\mu (t) = \sum_{k\geq 0}t^k \mu_k$ be the probability generating function of $\mu$. Since $\mu$ satisfies \eqref{eq:H_loc_mu}, we may apply Karamata's Abelian theorem \cite[Theorem 8.1.6]{Bingham_Goldie_Teugels_1987} (in the reference we take $n= \alpha = 1, \beta = 0$, $f_1(s) = G_\mu(e^{-s})-1+ms$, substitute $s$ by $-\log(s)$ and Taylor expand the logarithm) to write for $s \in (0,1]$
            \begin{equation}\label{eq: pgf}
                G_\mu(s) = 1- m(1-s) + (1-s) \ell(1-s)
            \end{equation}
            for a  function $\ell : (0,1] \rightarrow \R$ with $\ell(x) \sim \ell^\star(1/x)$ as $x \rightarrow 0$. We claim that for $s \in (0,1]$ we have
\begin{equation}\label{eq: l(s)}
                \ell(s) = m-\sum_{k\geq 0} \mu([k+1,\infty))(1-s)^k.
            \end{equation}
            This follows from \eqref{eq: pgf} by  writing $\ell(1-s)$ as
            \begin{eqnarray*}
                m + \frac{G_\mu(s)-1}{1-s} &=& m + \frac{1}{1-s} \sum_{k\geq 0} \mu_k (s^k-1) 
                  = m- \frac{1}{1-s} \sum_{k\geq 1}\mu_k (1-s)(s^{k-1}+ \cdots +1) \\
                & =& m-\sum_{k\geq 1} \sum_{j=0}^{k-1} \mu_j s^k=  m-\sum_{k\geq 0} \mu([k+1,\infty))s^k.
            \end{eqnarray*}
            From \eqref{eq: l(s)} it follows that $\ell(0):= \lim_{s \downarrow 0} = 0$, $\ell(1) = m-1+\mu_0 <m$ and $\ell(s)$ is increasing. Thus, $0<\ell(s)<m$ for every $s \in (0,1]$.
           
            Now,  by decomposing $ \mathcal{T}$ into the forest of subtrees rooted at children of the root we get
            \begin{equation*}
                Q_{n+1} = \sum_{k \geq 1} \mu_k \Big(1-\P(\mathsf{H}(\mathcal{T})< n)^k\Big) = \sum_{k\geq 1}\mu_k\Big(1-(1-Q_n)^k\Big) = 1-G_\mu(1-Q_n).
            \end{equation*}
            By substituting the expression of $G_{\mu}$ given by \eqref{eq: pgf}, we get (i).
            
            We turn to (ii). The upper bound simply comes from the fact that $\ell$ is positive, implying that  $Q_{n+1}  \leq  mQ_n$ for every $n \geq 0$. For the lower bound, fix $\eta \in (0,m)$. Since $\ell(x) \rightarrow 0$ as $ x \rightarrow 0$, we may choose $n_{0}$ such that $\ell(Q_{n}) \leq \eta/2$ for $n \geq n_{0}$. Then
$$
Q_{n+1} = Q_{n_0} \prod_{k=n_0}^n (m-\ell(Q_k)) \geq Q_{n_{0}} \left( m-\eta/2 \right)^{n-n_{0}+1}$$
which is at least $(m-\eta)^{n+1}$ for $n$ sufficiently large. This completes the proof.
            \end{proof}

We are now ready to finish the proof of Theorem \ref{thm:main}.

\begin{proof}[Proof of Theorem \ref{thm:main} (iii)]
We start with (a). Fix $\varepsilon \in (0,1)$. Take $h_{n}=(1+\varepsilon) \frac{\log(n)}{\log(1/m)}$. By Lemma \ref{lem:bounds} (ii), $n Q_{ h_{n}} \leq  {n^{-\varepsilon}}  \rightarrow 0$.
Now take $h_{n}=(1-\varepsilon) \frac{\log(n)}{\log(1/m)}$.  Fix  $\eta \in (0,m)$ such that
$$1-(1-\varepsilon) \frac{\log(m-\eta)}{\log m} \geq  \frac{\varepsilon}{2}.$$
Then by   Lemma \ref{lem:bounds} (ii), for $n$ sufficiently large we have $n Q_{ h_{n}} \geq n^{\varepsilon/2} \rightarrow \infty$. By applying Proposition \ref{prop:height}, the claim of (a) follows.

Now we turn to (b). Let $u_{n}>0$ be such that $Q_{n}= u_{n} m^{n }$.  By  \cite[Theorem 2]{HSV67}, $u_{n}$ converges to a positive constant  if $\sum_{n \geq 1} (n \log n) \mu_{n} <\infty$, and converges to $0$ otherwise.

If $\sum_{n \geq 1} (n \log n) \mu_{n} <\infty$, it follows that  $Q_{n} \sim c m^{n }$ as $n \rightarrow \infty$ for some $c>0$, which by Proposition \ref{prop:height} implies that for every sequence $\lambda_{n} \rightarrow \infty$ we have
$$ \P\left( \left|{\mathsf{H}(\mathcal{T}_n)}- \frac{\log(n)}{\log(1/m)} \right| \geq \lambda_{n}\right) \to 0,$$
implying that  the sequence $ \left({\mathsf{H}(\mathcal{T}_n)}- {\log(n)}/{\log(1/m)} \right)_{n \geq 1}$ is tight.

Now assume that  $\sum_{n \geq 1} (n \log n) \mu_{n} =\infty$, so that $u_{n} \rightarrow 0$. We build a sequence $r_{n} \rightarrow \infty$ such that
\begin{equation}
\label{eq:rn}
\P\left( {\mathsf{H}(\mathcal{T}_n)} \leq  \frac{\log(n)}{\log(1/m)} -r_{n}\right) \to 1,
\end{equation}
which will imply that the sequence $ \left({\mathsf{H}(\mathcal{T}_n)}- {\log(n)}/{\log(1/m)} \right)_{n \geq 1}$ is not tight. To this end, let $(a_{p})_{p \geq 1}$ be an increasing sequence of integers such that for every $p \geq 1$, for every $n \geq a_p$ we have $c_{n} \leq 1/p$. Then, let $(p_{n})_{n \geq 1}$ be the weakly increasing sequence of integers such that for every $n \geq 1$ we have
$$a_{p_{n}} \leq \frac{1}{2} \frac{\log(n)}{\log(1/m)} < a_{p_{n}+1}.$$
Observe that $p_{n} \rightarrow \infty$. Then set
$$r_{n}= \frac{1}{2 \log(1/m)} \min \left(  \log(n), \log({p_{n}}) \right) \quad \textrm{and} \quad  h_{n}=\frac{\log(n)}{\log(1/m)} -{r_{n}} .$$
Observe that $r_{n} \rightarrow \infty$ and that $h_{n} \geq  \frac{1}{2} \frac{\log(n)}{\log(1/m)} \geq a_{p_{n}}$. Thus $ u_{h_{n}} \leq 1/p_{n}$. As a consequence,
$$n Q_{ h_{n}} \leq n u_{h_n} m^{h_n} \leq n \cdot \frac{1}{p_{n}}  \cdot e^{-\log(n)+\log(1/m) r_n}    \leq  n \cdot \frac{1}{p_{n}}  \cdot \frac{\sqrt{p_n}}{n}  = \frac{1}{\sqrt{p_{n}}} \to 0.$$
The convergence \eqref{eq:rn}  then follows from Proposition \ref{prop:height} (ii) and this completes the proof.
\end{proof}

When $\sum_{n \geq 1}  (n \log n) \mu_{n}  =\infty$, in some particular cases it is possible to find the second order term in the magnitude of $\mathsf{H}(\mathcal{T}_{n})$ by analysing the asymptotic behavior of $u_n=Q_n/m^n$ as $n\rightarrow\infty$, as seen in the following result.

\begin{proposition}
\label{prop:ex}
Assume that $L(n) \sim c/\log(n)^{1+\beta}$ with $\beta \in (0,1]$.
\begin{enumerate}
\item[(i)] For $\beta=1$, for every $\varepsilon>0$ we have
$$ \P\left( \left| {\mathsf{H}(\mathcal{T}_n)}-  \frac{\log(n)}{\log(1/m)} + \frac{c}{ m \log(1/m)^2} \log \log n \right| \geq \varepsilon \log \log n\right) \to 0.$$
\item[(ii)] For $\beta \in (0,1)$, for every $\varepsilon>0$  we have
$$ \P\left( \left| {\mathsf{H}(\mathcal{T}_n)}-  \frac{\log(n)}{\log(1/m)} + \frac{c}{(1-\beta)\beta m \log(1/m)^{1+\beta}} \log(n)^{1-\beta} \right| \geq \varepsilon \log(n)^{1-\beta} \right) \to 0.$$
\end{enumerate}
\end{proposition}

\begin{proof}
Assume that $L(n) \sim c/\log(n)^{1+\beta}$ with $\beta \in (0,1]$. By Example \ref{ex:anl} (i), and using the fact $\ell(x)\sim\ell^\star(1/x)$ as $x \rightarrow 0$  (this was seen in the proof of Lemma \ref{lem:bounds}), the recurrence relation  $Q_{n+1}=Q_n(m-\ell(Q_n))$ can be rewritten as
\begin{equation}
    \label{eq:recrel1}
u_{n+1}=u_n \left(1- \frac{c/(\beta m)}{\left(\log(1/u_n)+n\log(1/m)\right)^\beta}(1+\varepsilon_n)\right)
\end{equation}
where $\varepsilon_n$ is a sequence going to $0$. Now we claim that $\log(1/u_n)=o(n)$. Indeed, by Lemma \ref{lem:bounds} for every $\eta \in (0,1)$ we have $\limsup ( \log(1/u_n)/n) \leq \log (m/(m-\eta))$ and by taking $\eta \rightarrow 0$ we get our claim. Thus the recurrence relation \eqref{eq:recrel1} can be rewritten as
\begin{equation}
    \label{eq:recrel2}
u_{n+1}=u_n \left(1- \frac{c/(\beta m)}{\left(n\log(1/m)\right)^\beta}(1+\varepsilon'_n)\right)
\end{equation}
where $\varepsilon'_n$ is a sequence going to $0$. 

Now assume that $\beta=1$. The recurrence relation \eqref{eq:recrel2} readily implies that  $\log(u_n) \sim  - \frac{c}{ m \log(1/m)} \log (n)$. Set
$$h^{\pm}_n= \frac{\log(n)}{\log(1/m)} - \frac{c}{ m \log(1/m)^2} \log \log n \pm  \varepsilon \log \log n$$
and observe that $\log(u_{h^\pm_n}) \sim - \frac{c}{ m \log(1/m)} \log \log n$ as $n\rightarrow \infty$. As a consequence,
$$\log \left(n Q_{ h^\pm_{n}}\right) =\log(u_{h^\pm_n}) +  \frac{c}{ m \log(1/m)} \log \log n -(\pm \varepsilon \log(1/m) \log \log n),$$
which is asymptotic to $\mp \varepsilon \log(1/m) \log \log n$ as $n \rightarrow \infty$. 
The conclusion then follows from Proposition \ref{prop:height}.

When $\beta \in (0,1)$, the recurrence relation \eqref{eq:recrel2} now implies that  we have $\log(u_n) \sim  - \frac{c}{(1-\beta)\beta m \log(1/m)^\beta} n^{1-\beta}$ and the desired result follows as in the case $\beta=1$.
\end{proof}

\providecommand{\bysame}{\leavevmode\hbox to3em{\hrulefill}\thinspace}
\providecommand{\MR}{\relax\ifhmode\unskip\space\fi MR }
% \MRhref is called by the amsart/book/proc definition of \MR.
\providecommand{\MRhref}[2]{%
  \href{http://www.ams.org/mathscinet-getitem?mr=#1}{#2}
}
\providecommand{\href}[2]{#2}

\end{document}